\documentclass[11pt,a4paper]{article}

\usepackage{amsmath}
\usepackage{amsthm}
\usepackage{amssymb}
\usepackage{amscd}
\usepackage[all]{xy}

\title{Vanishing Theorems on Toric Varieties \\
in Positive Characteristic
\footnote{This paper was partially supported by the National Natural Science Foundation of China
(Grant No.\ 11231003 and 11271070), and the Scientific Research Foundation for the Returned Overseas
Chinese Scholars, State Education Ministry.}}
\author{Qihong Xie}
\date{Dedicated to Professor Yujiro Kawamata for his sixtieth birthday}
\theoremstyle{plain}
\newtheorem{prop}{Proposition}[section]
\newtheorem{lem}[prop]{Lemma}
\newtheorem{thm}[prop]{Theorem}
\newtheorem{cor}[prop]{Corollary}

\theoremstyle{definition}
\newtheorem{defn}[prop]{Definition}
\newtheorem*{ack}{Acknowledgments}
\newtheorem*{nota}{Notation}

\theoremstyle{remark}

\newtheorem{ex}[prop]{Example}

\newcommand{\Q}{\mathbb Q}
\newcommand{\R}{\mathbb R}
\newcommand{\C}{\mathbb C}
\newcommand{\Z}{\mathbb Z}

\newcommand{\A}{\mathbb A}
\newcommand{\PP}{\mathbb P}
\newcommand{\OO}{\mathcal O}
\newcommand{\II}{\mathcal I}

\newcommand{\HH}{\mathcal H}
\newcommand{\LL}{\mathcal L}

\newcommand{\BB}{\mathcal B}
\newcommand{\ZZ}{\mathcal Z}

\newcommand{\Pic}{\mathop{\rm Pic}\nolimits}
\newcommand{\Div}{\mathop{\rm Div}\nolimits}

\newcommand{\Supp}{\mathop{\rm Supp}\nolimits}

\newcommand{\ch}{\mathop{\rm char}\nolimits}

\newcommand{\HOM}{\mathop{\mathcal Hom}\nolimits}

\newcommand{\spec}{\mathop{\rm Spec}\nolimits}

\newcommand{\divisor}{\mathop{\rm div}\nolimits}

\newcommand{\codim}{\mathop{\rm codim}\nolimits}
\newcommand{\ra}{\rightarrow}

\newcommand{\wt}{\widetilde}

\begin{document}

\maketitle

\begin{abstract}
We use the liftability of the relative Frobenius morphism of toric
varieties and the strong liftability of toric varieties to prove
the Bott vanishing theorem, the degeneration of the Hodge to de Rham
spectral sequence and the Kawamata-Viehweg vanishing theorem for
log pairs on toric varieties in positive characteristic. These results
generalize those results of Danilov, Buch-Thomsen-Lauritzen-Mehta,
Musta\c{t}\v{a} and Fujino to the case where concerned Weil divisors
are not necessarily torus invariant.
\end{abstract}

\setcounter{section}{0}
\section{Introduction}\label{S1}

Throughout this paper, we always work over {\it a perfect field $k$
of characteristic $p>0$} unless otherwise stated. The main purpose
of this paper is to develop various vanishing theorems on toric
varieties in positive characteristic by means of the lifting technique,
which consists of two points: one is the liftability of the relative
Frobenius morphism of toric varieties, and the other is the strong
liftability of toric varieties.

The following are the main theorems in this paper, which generalize
those results of Danilov \cite{da}, Buch-Thomsen-Lauritzen-Mehta
\cite{btlm}, Musta\c{t}\v{a} \cite{mu} and Fujino \cite{fu07} to the
case where concerned Weil divisors are not necessarily torus invariant.
See Definition \ref{2.12} for the definition of
$\wt{\Omega}^\bullet_X(\log D)$, the Zariski-de Rham complex of $X$
with logarithmic poles along $D$.

\begin{thm}[Bott vanishing]\label{1.1}
Let $X$ be a projective toric variety over $k$, $D$ a reduced Weil divisor
on $X$, and $\LL$ an ample invertible sheaf on $X$. Then
$H^j(X,\wt{\Omega}^i_X(\log D)\otimes\LL)=0$ holds for any $j>0$ and any $i\geq 0$.
\end{thm}

\begin{thm}[Hodge to de Rham spectral sequence]\label{1.2}
Let $X$ be a projective toric variety over $k$, and $D$ a reduced Weil
divisor on $X$. Then the Hodge to de Rham spectral sequence degenerates in $E_1$:
\begin{eqnarray*}
E_1^{ij}=H^j(X,\wt{\Omega}^i_X(\log D))\Longrightarrow
\mathbf{H}^{i+j}(X,\wt{\Omega}^\bullet_X(\log D)).
\end{eqnarray*}
\end{thm}

\begin{thm}[Kawamata-Viehweg vanishing]\label{1.3}
Let $X$ be a projective simplicial toric variety over $k$, and $H$ a nef and big 
$\Q$-divisor on $X$. Then $H^i(X,K_X+\ulcorner H\urcorner)=0$ holds for any $i>0$.
\end{thm}

All of the results are fresh for toric varieties in positive characteristic. 
Toric varieties are elementary, however, we can construct some interesting classes 
of algebraic varieties in positive characteristic from toric varieties. 
It turns out that the celebrated Kawamata-Viehweg vanishing theorem \cite{ka,vi} 
plays an essential role in birational geometry of complex algebraic varieties. 
Hence Theorem \ref{1.3} is also expected to be helpful in the study of 
algebraic varieties in positive characteristic.

Furthermore, we prove more general results as Theorems \ref{3.1},
\ref{3.2}, and \ref{3.4} for a normal projective variety
whose relative Frobenius morphism has a global lifting and which is
strongly liftable over $W_2(k)$. We also give an example as Corollary
\ref{3.7} to show that the main theorems could hold for more general
varieties.

In \S \ref{S2}, we will recall some definitions and preliminary results.
\S \ref{S3} is devoted to the proofs of the main theorems.
For the necessary notions and results in toric geometry,
we refer the reader to \cite{da}, \cite{od}, \cite{fu93} and \cite{clh}.

\begin{nota}
We use $[B]=\sum [b_i] B_i$ (resp.\ $\ulcorner B\urcorner=\sum \ulcorner
b_i\urcorner B_i$, $\langle B\rangle=\sum \langle b_i\rangle B_i$)
to denote the round-down (resp.\ round-up, fractional part)
of a $\Q$-divisor $B=\sum b_iB_i$, where for a real number $b$,
$[b]:=\max\{ n\in\Z \,|\,n\leq b \}$, $\ulcorner b\urcorner:=-[-b]$
and $\langle b\rangle:=b-[b]$.
\end{nota}

\begin{ack}
I would like to express my gratitude to Professors Luc Illusie and Osamu Fujino
for useful comments. I am very grateful to the referee for giving many useful 
suggestions, which make this paper more readable.
\end{ack}

\section{Preliminaries}\label{S2}

\begin{defn}\label{2.1}
The ring of Witt vectors of length two of $k$, denoted by $W_2(k)$,
is $k\oplus k$ as set, where addition and multiplication for
$a=(a_0,a_1)$ and $b=(b_0,b_1)$ in $W_2(k)$ are defined by
\begin{eqnarray*}
a+b &=& (a_0+b_0,a_1+b_1-\frac{1}{p}\sum_{0<i<p}{p\choose i}a_0^i
b_0^{p-i}), \\
a\cdot b &=& (a_0b_0,a_0^pb_1+b_0^pa_1).
\end{eqnarray*}
We can prove that $W_2(k)$ is flat over $\Z/p^2\Z$, $W_2(k)\otimes_
{\Z/p^2\Z}\Z/p\Z=k$ and $W_2(k)=\Z/p^2\Z$ if $k=\Z/p\Z$. In fact,
the projection $pr_1:W_2(k)\ra k$ given by $pr_1(a_0,a_1)=a_0$
corresponds to the reduction $W_2(k)/pW_2(k)=k$ modulo $p$, and the
ring homomorphism $F_{W_2(k)}:W_2(k)\ra W_2(k)$ given by $F_{W_2(k)}
(a_0,a_1)=(a_0^p,a_1^p)$ reduces to the Frobenius homomorphism
$F_k:k\ra k$ modulo $p$. We refer the reader to \cite[II.6]{se}
for more details.
\end{defn}

The following definition \cite[Definition 8.11]{ev} generalizes the
definition \cite[1.6]{di} of liftings of $k$-schemes over $W_2(k)$.

\begin{defn}\label{2.2}
Let $X$ be a noetherian scheme over $k$, and $D=\sum_i D_i$ a reduced
Cartier divisor on $X$. A lifting of $(X,D)$ over $W_2(k)$ consists of
a scheme $\wt{X}$ and closed subschemes $\wt{D}_i\subset\wt{X}$, all
defined and flat over $W_2(k)$ such that $X=\wt{X}\times_{\spec W_2(k)}
\spec k$ and $D_i=\wt{D}_i\times_{\spec W_2(k)}\spec k$. We write
$\wt{D}=\sum_i \wt{D}_i$ and say that $(\wt{X},\wt{D})$ is a lifting
of $(X,D)$ over $W_2(k)$, if no confusion is likely.
\end{defn}

Let $\wt{X}$ be a lifting of $X$ over $W_2(k)$. Then we have
an exact sequence of $\OO_{\wt{X}}$-modules:
\begin{eqnarray}
0\ra p\cdot\OO_{\wt{X}}\ra \OO_{\wt{X}}\stackrel{r}{\ra}
\OO_X\ra 0, \label{es1}
\end{eqnarray}
together with an $\OO_{\wt{X}}$-module isomorphism
\begin{eqnarray}
p:\OO_X\ra p\cdot\OO_{\wt{X}}, \label{es2}
\end{eqnarray}
where $r$ is the reduction modulo $p$ satisfying $p(x)=p\wt{x}$,
$r(\wt{x})=x$ for $x\in\OO_X$, $\wt{x}\in\OO_{\wt{X}}$.

Assume that $X$ is smooth over $k$ and $D=\sum_i D_i$ is simple normal
crossing. If $(\wt{X},\wt{D})$ is a lifting of $(X,D)$ over $W_2(k)$,
then $\wt{X}$ is smooth over $W_2(k)$ and $\wt{D}=\sum_i \wt{D}_i$ is
relatively simple normal crossing over $W_2(k)$, i.e.\ $\wt{X}$ is
covered by affine open subsets $\{U_\alpha\}$, such that each $U_\alpha$
is \'{e}tale over $\A^n_{W_2(k)}$ via coordinates $\{ x_1,\cdots,x_n \}$
and $\wt{D}|_{U_\alpha}$ is defined by the equation $x_1\cdots x_\nu=0$
with $1\leq\nu\leq n$ (see \cite[Lemmas 8.13, 8.14]{ev}). Therefore
we have an exact sequence of $\OO_{\wt{X}}$-modules:
\begin{eqnarray*}
0\ra p\cdot\Omega^1_{\wt{X}/W_2(k)}(\log \wt{D})\ra \Omega^1_{\wt{X}
/W_2(k)}(\log \wt{D})\stackrel{r}{\ra} \Omega^1_X(\log D)\ra 0,
\end{eqnarray*}
together with an $\OO_{\wt{X}}$-module isomorphism
$p:\Omega^1_X(\log D)\ra p\cdot\Omega^1_{\wt{X}/W_2(k)}(\log \wt{D})$.

\begin{defn}\label{2.3}
Let $X$ be a noetherian scheme over $k$, $F_X:X\ra X$ the absolute
Frobenius morphism of $X$, and $X'=X\times_{\spec k} F_k$ the fiber
product. Then we have the following commutative diagram with cartesian
square:
\[
\xymatrix{
X \ar[dr] \ar[r]_{F} \ar@/^1pc/[rr]^{F_X} & X' \ar[d] \ar[r] & X \ar[d] \\
 & \spec k \ar[r]^{F_k} & \spec k.
}
\]
The induced $k$-morphism $F:X\ra X'$ is called the relative Frobenius
morphism of $X/k$. Assume that $X$ has a lifting $\wt{X}$ over $W_2(k)$.
Define $\wt{X}'=\wt{X}\times_{\spec W_2(k)}F_{W_2(k)}$. A $W_2(k)$-morphism
$\wt{F}:\wt{X}\ra \wt{X}'$ is called a lifting of the relative Frobenius
morphism of $X/k$, if the restriction of $\wt{F}$ to $X$ is just $F:X\ra X'$.
\[
\xymatrix{
\wt{X} \ar[dr] \ar[r]^{\wt{F}} & \wt{X}' \ar[d] \ar[r] & \wt{X} \ar[d] \\
 & \spec W_2(k) \ar[r]^{F_{W_2(k)}} & \spec W_2(k).
}
\]

Let $D$ be a reduced Cartier divisor on $X$, $\wt{D}$ a lifting of $D$ over
$W_2(k)$ and $\wt{D}'=\wt{D}\times_{\spec W_2(k)}F_{W_2(k)}$ the divisor on
$\wt{X}'$. $\wt{F}:\wt{X}\ra \wt{X}'$ is said to be compatible with $\wt{D}$,
if $\wt{F}^*\OO_{\wt{X}'}(-\wt{D}')=\OO_{\wt{X}}(-p\wt{D})$ holds.
\end{defn}

The following theorem is essentially due to Deligne and Illusie, although its 
explicit statement and proof have not been given in \cite{di}.

\begin{thm}\label{2.4}
Let $X$ be a smooth scheme over $k$, and $D=\sum_i D_i$ a simple normal
crossing divisor on $X$. Assume that $(X,D)$ has a lifting $(\wt{X},\wt{D})$
over $W_2(k)$ and the relative Frobenius morphism $F:X\ra X'$ has a lifting
$\wt{F}:\wt{X}\ra \wt{X}'$, which is compatible with $\wt{D}$. Then there is
a quasi-isomorphism of complexes of $\OO_{X'}$-modules:
\begin{eqnarray*}
\phi:\bigoplus_i\Omega^i_{X'}(\log D')[-i]\ra F_*\Omega^\bullet_X(\log D).
\end{eqnarray*}
\end{thm}

\begin{proof}
We use a similar proof to that of \cite[Th\'eor\`eme 2.1]{di}. Since
$F^*:\Omega^1_{X'}(\log D')\ra F_*\Omega^1_X(\log D)$ is trivial, the image of
$\wt{F}^*:\Omega^1_{\wt{X}'/W_2(k)}(\log \wt{D}')\ra \wt{F}_*\Omega^1_{\wt{X}/
W_2(k)}(\log \wt{D})$ is contained in $p\cdot\wt{F}_*\Omega^1_{\wt{X}/W_2(k)}
(\log \wt{D})$. Therefore, there exists a unique homomorphism $f=p^{-1}\wt{F}^*:
\Omega^1_{X'}(\log D')\ra F_*\Omega^1_X(\log D)$ making the following diagram
commutative:
\[
\xymatrix{
\Omega^1_{\wt{X}'/W_2(k)}(\log \wt{D}') \ar@{->>}[d]_r \ar[r]^{\wt{F}^*} &
p\cdot\wt{F}_*\Omega^1_{\wt{X}/W_2(k)}(\log \wt{D}) \\
\Omega^1_{X'}(\log D') \ar[r]^{f} & F_*\Omega^1_X(\log D) \ar[u]^{\cong}_p.
}
\]
Let $\wt{x}$ be a local section of $\OO_{\wt{X}}(-\wt{D})$ and $x=r(\wt{x})$
the induced local section of $\OO_X(-D)$. Since $\wt{F}$ is compatible with
$\wt{D}$, we have $\wt{F}^*(\wt{x}\otimes 1)=\wt{x}^p(1+pu(\wt{x}))$, where
$u(\wt{x})$ is a local section of $\OO_{\wt{X}}$. Hence we have
\begin{eqnarray}
f(\frac{dx\otimes 1}{x\otimes 1})=\frac{dx}{x}+du(\wt{x}). \label{es3}
\end{eqnarray}
In particular, we have $df=0$, hence $f$ induces a homomorphism $f:\Omega^1
_{X'}(\log D')\ra \ZZ^1(F_*\Omega^\bullet_X(\log D))$. Define $\phi^0:\OO_{X'}
\ra F_*\OO_X$ to be the Frobenius homomorphism. For any $i\geq 1$, define
$\phi^i=\wedge^i f:\Omega^i_{X'}(\log D')\ra \ZZ^i(F_*\Omega^\bullet_X
(\log D))$ to be the $i$-th wedge product of $f$. By abuse of notation,
we denote $\phi^i:\Omega^i_{X'}(\log D')\ra F_*\Omega^i_X(\log D)$ to be
the composition of $\wedge^i f$ with the natural inclusion
$\ZZ^i(F_*\Omega^\bullet_X(\log D))\hookrightarrow F_*\Omega^i_X(\log D)$.
Thus we have a morphism of complexes of $\OO_{X'}$-modules:
\begin{eqnarray*}
\phi=\bigoplus_i\phi^i:\bigoplus_i\Omega^i_{X'}(\log D')[-i]\ra
F_*\Omega^\bullet_X(\log D).
\end{eqnarray*}
It follows from (\ref{es3}) and the definition of $\phi^i$ that
$\HH^i(\phi)=C^{-1}$ holds for any $i\geq 0$, where $C:\HH^i(F_*\Omega
^\bullet_X(\log D))\ra \Omega^i_{X'}(\log D')$ is the Cartier isomorphism
(cf.\ \cite[7.2]{katz}). Thus $\phi$ is a quasi-isomorphism of complexes
of $\OO_{X'}$-modules.
\end{proof}

Deligne and Illusie \cite{di} proved that if $(X,D)$ has a lifting over $W_2(k)$ 
and $p=\ch(k)>\dim X$, then the de Rham complex $F_*\Omega^\bullet_X(\log D)$ is 
decomposable in derived category, namely, the isomorphism $\phi$ described in
Theorem \ref{2.4} exists in the derived category. On the other hand, Theorem \ref{2.4} 
proves that if we assume the liftability of the relative Frobenius morphism 
instead of the condition $p>\dim X$, then the map $\phi$ described in Theorem \ref{2.4} 
exists in the category of complexes, which illustrates that the lifting condition of 
the relative Frobenius morphism is indeed very strong.

Next, we can generalize \cite[Theorem 2]{btlm} to the logarithmic case.

\begin{prop}\label{2.5}
Notation and assumptions as in Theorem \ref{2.4}, then the homomorphism
$\phi^i:\Omega^i_{X'}(\log D')\ra F_*\Omega^i_X(\log D)$ is a split
injection for any $i\geq 0$.
\end{prop}

\begin{proof}
Denote $n=\dim X$, $\BB^i=\BB^i(F_*\Omega^\bullet_X(\log D))$, $\ZZ^i=\ZZ^i
(F_*\Omega^\bullet_X(\log D))$. By using the Cartier isomorphism
\cite[7.2]{katz}, we have the following exact sequence for any $i\geq 0$:
\begin{eqnarray}
0\ra \BB^i\ra \ZZ^i\stackrel{C}{\ra} \Omega^i_{X'}(\log D')\ra 0. \label{es4}
\end{eqnarray}
In fact, the exact sequence (\ref{es4}) splits since $C\phi^i=\mathrm{Id}$
holds for any $i\geq 0$.

Define a homomorphism $\psi^i:F_*\Omega^i_X(\log D)\ra \HOM_{\OO_{X'}}
(\Omega^{n-i}_{X'}(\log D'),\Omega^n_{X'}(\log D'))$ by $\omega\mapsto\psi^i
(\omega)$, where $\psi^i(\omega)(\eta)=C(\phi^{n-i}(\eta)\wedge\omega)$ and
$C:F_*\Omega^n_X(\log D)=\ZZ^n\ra \Omega^n_{X'}(\log D')$ is the homomorphism
in the exact sequence (\ref{es4}) for $i=n$. Let $z$ be a local section of
$\Omega^i_{X'}(\log D')$. Then we have $\psi^i(\phi^i(z))(\eta)=C(\phi^{n-i}
(\eta)\wedge\phi^i(z))=C(\phi^n(\eta\wedge z))=\eta\wedge z$. By using the
perfect pairing between $\Omega^{n-i}_{X'}(\log D')$ and $\Omega^i_{X'}
(\log D')$ given by the wedge product, we obtain a well-defined homomorphism
$\psi^i:F_*\Omega^i_X(\log D)\ra \Omega^i_{X'}(\log D')$ satisfying
$\psi^i\phi^i=\mathrm{Id}$, which is the desired splitting of $\phi^i$.
\end{proof}

In order to deal with $\Q$-divisors, we need the following lemma due to Hara \cite{ha}.

\begin{lem}\label{2.6}
Let $X$ be a smooth scheme over $k$, $D=\sum_i D_i$ a simple normal crossing
divisor on $X$, and $G=\sum_ir_iD_i$ an effective divisor on $X$ with
$0\leq r_i<p$ for any $i$. Then the inclusion of complexes $\Omega^\bullet
_X(\log D)\hookrightarrow\Omega^\bullet_X(\log D)(G):=\Omega^\bullet_X
(\log D)\otimes\OO_X(G)$ induces a quasi-isomorphism of complexes of
$\OO_{X'}$-modules:
\begin{eqnarray*}
\nu:F_*\Omega^\bullet_X(\log D)\hookrightarrow F_*(\Omega^\bullet_X
(\log D)(G)).
\end{eqnarray*}
\end{lem}

\begin{proof}
By K\"unneth's formula, we can reduce to prove the case where $\dim X=1$,
$D$ is defined by the local parameter $t$ and $G=rD$ with $0\leq r<p$.
In this case, the complex of $\OO_{X'}$-modules
$F_*(\Omega^\bullet_X(\log D)(G))$ is written as:
\[
0\ra \bigoplus_{i=0}^{p-1}\OO_{X'}\cdot t^i\cdot\frac{1}{t^r}
\stackrel{d}{\ra}\bigoplus_{i=0}^{p-1}\OO_{X'}\cdot \frac{dt}{t}
\cdot t^i\cdot\frac{1}{t^r}\ra 0.
\]
By an easy calculation, we have $\HH^0=\OO_{X'}$ and
$\HH^1=\OO_{X'}\cdot dt/t$, which is independent of $r$.
Hence $\nu$ is a quasi-isomorphism.
\end{proof}

With the same notation and assumptions as in Theorem \ref{2.4} and
Lemma \ref{2.6}, define $\phi^i_G$ to be the composition of the
following homomorphisms for any $i\geq 0$:
\begin{eqnarray*}
\phi^i_G:\Omega^i_{X'}(\log D')\stackrel{\phi^i}{\ra} \ZZ^i(F_*\Omega^\bullet
_X(\log D))\stackrel{\nu}{\ra} \ZZ^i(F_*(\Omega^\bullet_X(\log D)(G)))
\hookrightarrow F_*(\Omega^i_X(\log D)(G)).
\end{eqnarray*}
Since $\phi^i\wedge\phi^j=\phi^{i+j}$ holds for any $i,j\geq 0$, by using
the natural homomorphism $\wedge:\Omega^i_X(\log D)\otimes\Omega^j_X(\log D)
(G)\ra \Omega^{i+j}_X(\log D)(G)$, we can show that $\phi^i\wedge\phi^j_G
=\phi^{i+j}_G$ holds for any $i,j\geq 0$. Denote the composition of
$\HH^i(\nu)^{-1}:\HH^i(F_*\Omega^\bullet_X(\log D)(G))\ra \HH^i(F_*\Omega
^\bullet_X(\log D))$ with the Cartier isomorphism $C:\HH^i(F_*\Omega^\bullet
_X(\log D))\ra \Omega^i_{X'}(\log D')$ by $C_G:\HH^i(F_*(\Omega^\bullet_X
(\log D)(G)))\ra \Omega^i_{X'}(\log D')$. Then the morphism
\begin{eqnarray*}
\phi_G=\bigoplus_i\phi^i_G:\bigoplus_i\Omega^i_{X'}(\log D')[-i]\ra
F_*(\Omega^\bullet_X(\log D)(G))
\end{eqnarray*}
is a quasi-isomorphism since it induces the isomorphism $C^{-1}_G=\HH^i
(\nu)C^{-1}$ on each $\HH^i$.

\begin{prop}\label{2.7}
Let $X$ be a smooth scheme over $k$, $D=\sum_i D_i$ a simple normal crossing
divisor on $X$, and $G=\sum_ir_iD_i$ an effective divisor on $X$ with
$0\leq r_i<p$ for any $i$. Assume that $(X,D)$ has a lifting $(\wt{X},\wt{D})$
over $W_2(k)$ and the relative Frobenius morphism $F:X\ra X'$ has a lifting
$\wt{F}:\wt{X}\ra \wt{X}'$, which is compatible with $\wt{D}$.
Then $\phi^i_G:\Omega^i_{X'}(\log D')\ra F_*(\Omega^i_X(\log D)(G))$
is a split injection for any $i\geq 0$.
\end{prop}

\begin{proof}
Denote $n=\dim X$, $\BB^i_G=\BB^i(F_*(\Omega^\bullet_X(\log D)(G)))$,
$\ZZ^i_G=\ZZ^i(F_*(\Omega^\bullet_X(\log D)(G)))$. By using the
isomorphism $C_G$, we have the following exact sequence for any $i\geq 0$:
\begin{eqnarray}
0\ra \BB^i_G\ra \ZZ^i_G\stackrel{C_G}{\ra} \Omega^i_{X'}(\log D')\ra 0.
\label{es5}
\end{eqnarray}
In fact, the exact sequence (\ref{es5}) splits since $C_G\phi^i_G=
\mathrm{Id}$ holds for any $i\geq 0$.

Define a homomorphism
\begin{eqnarray*}
\psi^i_G:F_*(\Omega^i_X(\log D)(G))\ra \HOM_{\OO_{X'}}(\Omega^{n-i}_{X'}
(\log D'),\Omega^n_{X'}(\log D'))
\end{eqnarray*}
by $\omega\mapsto\psi^i_G(\omega)$, where $\psi^i_G(\omega)(\eta)=C_G
(\phi^{n-i}(\eta)\wedge\omega)$ and $C_G:F_*(\Omega^n_X(\log D)(G))=
\ZZ^n_G\ra \Omega^n_{X'}(\log D')$ is the homomorphism in the exact
sequence (\ref{es5}) for $i=n$. Let $z$ be a local section of
$\Omega^i_{X'}(\log D')$. Then we have $\psi^i_G(\phi^i_G(z))(\eta)
=C_G(\phi^{n-i}(\eta)\wedge\phi^i_G(z))=C_G(\phi^n_G(\eta\wedge z))
=\eta\wedge z$. By using the perfect pairing between $\Omega^{n-i}_{X'}
(\log D')$ and $\Omega^i_{X'}(\log D')$ given by the wedge product,
we obtain a well-defined homomorphism
$\psi^i_G:F_*(\Omega^i_X(\log D)(G))\ra \Omega^i_{X'}(\log D')$
satisfying $\psi^i_G\phi^i_G=\mathrm{Id}$, which is the desired
splitting of $\phi^i_G$.
\end{proof}

\begin{thm}\label{2.8}
Notation and assumptions as in Proposition \ref{2.7}, let $H$ be a
$\Q$-divisor on $X$ with $\Supp(\langle H\rangle)\subseteq D$. Then
there exists a quasi-isomorphism of complexes of $\OO_{X'}$-modules:
\begin{eqnarray*}
\phi_H=\bigoplus_i\phi^i_H:\bigoplus_i\Omega^i_{X'}(\log D')(-\ulcorner
H'\urcorner)[-i]\ra F_*(\Omega^\bullet_X(\log D)(-\ulcorner pH\urcorner)).
\end{eqnarray*}
In fact, $\phi^i_H:\Omega^i_{X'}(\log D')(-\ulcorner H'\urcorner)\ra
F_*(\Omega^i_X(\log D)(-\ulcorner pH\urcorner))$ is a split injection
for any $i\geq 0$.
\end{thm}

\begin{proof}
Let $G=p\ulcorner H\urcorner-\ulcorner pH\urcorner$. Then it is easy to
see that $G$ satisfies the condition in Proposition \ref{2.7}. Tensoring
$\phi^i_G$ by $\OO_{X'}(-\ulcorner H'\urcorner)$, we can obtain the
conclusions.
\end{proof}

There is a slight generalization of \cite[Definition 2.3]{xie}.

\begin{defn}\label{2.9}
Let $X$ be a noetherian scheme over $k$. $X$ is said to be strongly
liftable over $W_2(k)$, if the following two conditions hold:
\begin{itemize}
\item[(i)] $X$ has a lifting $\wt{X}$ over $W_2(k)$;
\item[(ii)] For any effective Cartier divisor $D$ on $X$, there is a lifting
$\wt{D}$ of $D$ over $W_2(k)$ in the following sense: $D$ is regarded as
a closed subscheme of $X$, and $\wt{D}$ is a closed subscheme of the fixed
scheme $\wt{X}$ such that $\wt{D}$ is flat over $W_2(k)$ and $\wt{D}\times
_{\spec W_2(k)}\spec k=D$.
\end{itemize}
\end{defn}

If $X$ is a smooth scheme over $k$, then Definition \ref{2.9} is equivalent
to \cite[Definition 2.3]{xie}. Furthermore, it was proved in \cite{xie,xie11} 
that $\A^n_k$, $\PP^n_k$, smooth projective curves, smooth projective rational 
surfaces, certain smooth complete intersections in $\PP^n_k$, and smooth toric 
varieties are strongly liftable over $W_2(k)$.

Let $X$ be a noetherian scheme over $k$, $\wt{X}$ a lifting of $X$ over
$W_2(k)$, $\LL$ an invertible sheaf on $X$ and $\wt{\LL}$ a lifting of
$\LL$ on $\wt{X}$, i.e.\ $\wt{\LL}$ is an invertible sheaf on $\wt{X}$
such that $\wt{\LL}|_X=\LL$. Tensoring (\ref{es1}) by $\wt{\LL}$ and
taking cohomology groups, we obtain the following exact sequence:
\begin{eqnarray}
0\ra H^0(\wt{X},p\cdot\wt{\LL})\ra H^0(\wt{X},\wt{\LL})\stackrel{r}{\ra}
H^0(X,\LL)\ra H^1(\wt{X},p\cdot\wt{\LL}). \label{es6}
\end{eqnarray}

We recall here a sufficient condition for strong liftability of schemes
(cf.\ \cite[Proposition 3.4]{xie}).

\begin{prop}\label{2.10}
Let $X$ be a noetherian scheme over $k$, and $\wt{X}$ a lifting of $X$ over
$W_2(k)$. Assume that for any effective Cartier divisor $D$ on $X$, the
associated invertible sheaf $\LL=\OO_X(D)$ has a lifting $\wt{\LL}$ on
$\wt{X}$, and the natural map $r:H^0(\wt{X},\wt{\LL})\ra H^0(X,\LL)$
is surjective. Then $X$ is strongly liftable over $W_2(k)$.
\end{prop}

\begin{proof}
Let $D$ be an effective Cartier divisor on $X$, $\LL=\OO_X(D)$ the associated
invertible sheaf on $X$, and $s\in H^0(X,\LL)$ the corresponding nonzero
section. By assumption, $\LL$ has a lifting $\wt{\LL}$ on $\wt{X}$, and
the natural map $r:H^0(\wt{X},\wt{\LL})\ra H^0(X,\LL)$ is surjective,
hence there is a nonzero section $\wt{s}\in H^0(\wt{X},\wt{\LL})$ with
$r(\wt{s})=s$. Define $\wt{D}=\divisor_0(\wt{s})$ to be the associated
effective Cartier divisor on $\wt{X}$. Then it is easy to see that $\wt{D}$
is a lifting of $D$ over $W_2(k)$.
\end{proof}

Let $X$ be a normal scheme over $k$, and $D$ a reduced Weil divisor on $X$.
Then there exists an open subset $U$ of $X$ such that $\codim_X(X-U)\geq 2$,
$U$ is smooth over $k$, and $D|_U$ is simple normal crossing on $U$.
Such an $U$ is called a required open subset for the log pair $(X,D)$.

\begin{ex}\label{2.11}
(i) Let $X$ be a smooth scheme over $k$, and $D$ a simple normal crossing
divisor on $X$. Then we can take the largest required open subset $U=X$.

(ii) Let $X=X(\Delta,k)$ be a toric variety, and $D$ a torus invariant reduced
Weil divisor on $X$. Take $U$ to be $X(\Delta_1,k)$, where $\Delta_1$ is the fan
consisting of all 1-dimensional cones in $\Delta$. Then it is easy to show that
$U$ is an open subset of $X$ with $\codim_X(X-U)\geq 2$, $U$ is smooth over $k$,
and $D|_U$ is simple normal crossing on $U$. Hence $U$ is a required open subset
for $(X,D)$.
\end{ex}

\begin{defn}\label{2.12}
Let $X$ be a normal scheme over $k$, and $D$ a reduced Weil divisor on $X$.
Take a required open subset $U$ for $(X,D)$, and denote $\iota:U\hookrightarrow X$
to be the open immersion. For any $i\geq 0$, define the Zariski sheaf of differential
$i$-forms of $X$ with logarithmic poles along $D$ by
\begin{eqnarray*}
\wt{\Omega}^i_X(\log D)=\iota_*\Omega^i_U(\log D|_U).
\end{eqnarray*}
Since $\codim_X(X-U)\geq 2$ and $\Omega^i_U(\log D|_U)$ is a locally free
sheaf on $U$, $\wt{\Omega}^i_X(\log D)$ is a reflexive sheaf on $X$, and
the definition of $\wt{\Omega}^i_X(\log D)$ is independent of the choice
of the open subset $U$. Furthermore we have the Zariski-de Rham complex
$(\wt{\Omega}^\bullet_X(\log D),d)$.
\end{defn}

\section{Proofs of the main theorems}\label{S3}

First of all, we will prove some more general results, which imply
the main theorems.

\begin{thm}\label{3.1}
Let $X$ be a normal projective variety over $k$, $D$ a reduced Weil divisor
on $X$, and $U$ a required open subset for $(X,D)$. Assume that $(U,D|_U)$
has a lifting $(\wt{U},\wt{D|_U})$ over $W_2(k)$ and $F:U\ra U'$ has a lifting
$\wt{F}:\wt{U}\ra \wt{U}'$, which is compatible with $\wt{D|_U}$. Then
$H^j(X,\wt{\Omega}^i_X(\log D)\otimes\LL)=0$ holds for any $j>0$, any $i\geq 0$
and any ample invertible sheaf $\LL$ on $X$.
\end{thm}

\begin{proof}
By assumption and Proposition \ref{2.5}, we have a split injection
$\phi^i:\Omega^i_{U'}(\log D'|_{U'})\ra F_*\Omega^i_U(\log D|_U)$ for any
$i\geq 0$. Let $\iota:U\hookrightarrow X$ and $\iota':U'\hookrightarrow X'$
be the open immersions. Applying $\iota'_*$ to $\phi^i$, since
$\iota'_*F_*=F_*\iota_*$, we have a split injection for any $i\geq 0$:
\begin{eqnarray*}
\iota'_*\phi^i:\wt{\Omega}^i_{X'}(\log D')\ra F_*\wt{\Omega}^i_X(\log D).
\end{eqnarray*}
Tensoring $\iota'_*\phi^i$ by $\LL'$ and using the projection formula,
since $k$ is perfect, we obtain an injection for any $j>0$:
\begin{eqnarray*}
H^j(X,\wt{\Omega}^i_X(\log D)\otimes\LL)=H^j(X',\wt{\Omega}^i_{X'}(\log D')
\otimes\LL')\hookrightarrow H^j(X,\wt{\Omega}^i_X(\log D)\otimes\LL^p).
\end{eqnarray*}
Iterating these injections, we obtain the desired vanishing by Serre's
vanishing.
\end{proof}

Associated with the Zariski-de Rham complex $\wt{\Omega}^\bullet_X(\log D)$,
there is a spectral sequence:
\begin{eqnarray}
E^{ij}_1=H^j(X,\wt{\Omega}^i_X(\log D))\Longrightarrow \mathbf{H}^{i+j}
(X,\wt{\Omega}^\bullet_X(\log D)), \label{es7}
\end{eqnarray}
where $\mathbf{H}^\bullet(X,\wt{\Omega}^\bullet_X(\log D))$ denotes the
hypercohomology of the complex $\wt{\Omega}^\bullet_X(\log D)$. (\ref{es7})
is called the Hodge to de Rham spectral sequence for the Zariski-de Rham
complex $\wt{\Omega}^\bullet_X(\log D)$.

\begin{thm}\label{3.2}
Notation and assumptions as in Theorem \ref{3.1}, then the Hodge to de Rham
spectral sequence (\ref{es7}) degenerates in $E_1$.
\end{thm}

\begin{proof}
By Theorem \ref{2.4} and Proposition \ref{2.5}, there is a split injection of
complexes:
\begin{eqnarray*}
\phi:\bigoplus_i\Omega^i_{U'}(\log D'|_{U'})[-i]\ra F_*\Omega^\bullet_U
(\log D|_U).
\end{eqnarray*}
Applying $\iota'_*$ to $\phi$, since $\iota'_*F_*=F_*\iota_*$, we have
a split injection of complexes:
\begin{eqnarray*}
\iota'_*\phi:\bigoplus_i\wt{\Omega}^i_{X'}(\log D')[-i]\ra F_*\wt{\Omega}
^\bullet_X(\log D).
\end{eqnarray*}
Thus we have
\begin{eqnarray*}
\sum_{i+j=n}\dim_k E^{ij}_\infty &=& \dim_k \mathbf{H}^n(X,\wt{\Omega}
^\bullet_X(\log D))=\dim_k \mathbf{H}^n(X',F_*\wt{\Omega}^\bullet_X
(\log D)) \\
&\geq& \dim_k \mathbf{H}^n(X',\bigoplus_i\wt{\Omega}^i_{X'}(\log D')[-i])
=\sum_{i+j=n}\dim_k H^j(X',\wt{\Omega}^i_{X'}(\log D')) \\
&=& \sum_{i+j=n}\dim_k H^j(X,\wt{\Omega}^i_{X}(\log D))
=\sum_{i+j=n}\dim_k E^{ij}_1.
\end{eqnarray*}
Since $E^{ij}_\infty$ is a subquotient of $E^{ij}_1$, we have
$\sum_{i+j=n}\dim_k E^{ij}_\infty\leq \sum_{i+j=n}\dim_k E^{ij}_1$,
which implies $E^{ij}_\infty\cong E^{ij}_1$ for any $i,j$, i.e.\
the Hodge to de Rham spectral sequence degenerates in $E_1$.
\end{proof}

\begin{lem}\label{3.3}
Notation and assumptions as in Theorem \ref{3.1}, let $H$ be an ample
$\Q$-divisor on $X$ such that $\Supp(\langle H\rangle)\subseteq D$.
Then there is an injection for any $r>0$ and any $i,j\geq 0$:
\[
H^j(X,\wt{\Omega}^i_X(\log D)(-\ulcorner H\urcorner))\hookrightarrow
H^j(X,\wt{\Omega}^i_X(\log D)(-\ulcorner p^rH\urcorner)).
\]
\end{lem}

\begin{proof}
By Theorem \ref{2.8}, we have a split injection for any $i\geq 0$:
\[ \phi^i_H:\Omega^i_{U'}(\log D'|_{U'})(-\ulcorner H'|_{U'}\urcorner)
\ra F_*(\Omega^i_U(\log D|_U)(-\ulcorner pH|_U\urcorner)).
\]
Applying $\iota'_*$ to $\phi^i_H$, since $\iota'_*F_*=F_*\iota_*$,
we have a split injection for any $i\geq 0$:
\begin{eqnarray*}
\iota'_*\phi^i_H:\wt{\Omega}^i_{X'}(\log D')(-\ulcorner H'\urcorner)\ra
F_*(\wt{\Omega}^i_X(\log D)(-\ulcorner pH\urcorner)),
\end{eqnarray*}
which gives rise to an injection for any $i,j\geq 0$:
\begin{eqnarray*}
H^j(X,\wt{\Omega}^i_X(\log D)(-\ulcorner H\urcorner))\hookrightarrow
H^j(X,\wt{\Omega}^i_X(\log D)(-\ulcorner pH\urcorner)).
\end{eqnarray*}
Iterating these injections, we can obtain an injection for any $r>0$
and any $i,j\geq 0$:
\[
H^j(X,\wt{\Omega}^i_X(\log D)(-\ulcorner H\urcorner))\hookrightarrow
H^j(X,\wt{\Omega}^i_X(\log D)(-\ulcorner p^rH\urcorner)).
\]
\end{proof}

\begin{cor}\label{3.8}
Let $X$ be a smooth projective variety over $k$, $D$ a simple normal
crossing divisor on $X$, and $H$ an ample $\Q$-divisor on $X$ such that
$\Supp(\langle H\rangle)\subseteq D$. Assume that $(X,D)$ has a lifting
$(\wt{X},\wt{D})$ over $W_2(k)$ and $F:X\ra X'$ has a lifting
$\wt{F}:\wt{X}\ra \wt{X}'$, which is compatible with $\wt{D}$. Then
\[
H^j(X,\Omega^i_X(\log D)(-\ulcorner H\urcorner))=0
\]
holds for any $j<\dim X$ and any $i\geq 0$.
\end{cor}

\begin{proof}
By Lemma \ref{3.3}, there is an injection for any $r>0$, any $i\geq 0$
and any $j<\dim X$:
\begin{eqnarray}
H^j(X,\Omega^i_X(\log D)(-\ulcorner H\urcorner))\hookrightarrow
H^j(X,\Omega^i_X(\log D)(-\ulcorner p^rH\urcorner)). \label{es8}
\end{eqnarray}
Assume that $mH$ is a Cartier divisor for some positive integer $m$.
Assume $p^r=sm+t$, where $0\leq t<m$. Then $\ulcorner p^rH\urcorner=
s(mH)+\ulcorner tH\urcorner$. We can take $r$ sufficiently large,
such that $s$ is sufficiently large. Since $\Omega^i_X(\log D)(-\ulcorner
tH\urcorner)$ is locally free,
\begin{eqnarray}
H^j(X,\Omega^i_X(\log D)(-\ulcorner p^rH\urcorner))=H^j(X,\Omega^i_X(\log D)
(-\ulcorner tH\urcorner)\otimes\OO_X(mH)^{-s})=0 \label{es9}
\end{eqnarray}
holds for any $0\leq t<m$ and any $j<\dim X$ by Serre's duality and Serre's vanishing.
Thus the injection (\ref{es8}) implies the desired vanishing.
\end{proof}

It follows from \cite[Theorem 5.71 and Corollary 5.72]{km} that Serre's duality
and the vanishing (\ref{es9}) hold only for Cohen-Macaulay sheaves on normal projective
varieties. In general, the reflexive sheaves $\wt{\Omega}^i_X(\log D)(-\ulcorner tH\urcorner)$
are not necessarily Cohen-Macaulay, hence we cannot use the same argument to get a generalization
of Corollary \ref{3.8} to normal projective varieties. In order to deal with the Kawamata-Viehweg
vanishing, it is helpful to introduce the following notion.

\begin{defn}\label{3.9}
Let $X$ be a normal projective variety of dimension $n$ over $k$. We say that Serre's
duality holds for any Weil divisor on $X$, if there is an isomorphism of $k$-vector spaces
for any Weil divisor $D$ on $X$ and any $i\geq 0$:
\[
H^i(X,\OO_X(D))^\vee\stackrel{\sim}{\longrightarrow} H^{n-i}(X,\OO_X(K_X-D)),
\]
where $K_X$ is the canonical divisor of $X$ (see \cite[Proposition 5.75]{km}).
\end{defn}

\begin{thm}\label{3.4}
Let $X$ be a normal projective variety over $k$, $D$ a reduced Weil divisor
on $X$, and $U$ a required open subset for $(X,D)$. Assume that $(U,D|_U)$
has a lifting $(\wt{U},\wt{D|_U})$ over $W_2(k)$ and $F:U\ra U'$ has a lifting
$\wt{F}:\wt{U}\ra \wt{U}'$, which is compatible with $\wt{D|_U}$.
Let $H$ be an ample $\Q$-divisor on $X$ such that $\Supp(\langle H\rangle)\subseteq D$.
Assume that Serre's duality holds for any Weil divisor on $X$.
Then $H^j(X,K_X+\ulcorner H\urcorner)=0$ holds for any $j>0$.
\end{thm}

\begin{proof}
Taking $i=0$ in Lemma \ref{3.3}, we have an injection for any $r>0$
and any $j<\dim X$:
\begin{eqnarray}
H^j(X,\OO_X(-\ulcorner H\urcorner))\hookrightarrow
H^j(X,\OO_X(-\ulcorner p^rH\urcorner)). \label{es10}
\end{eqnarray}
By assumption, Serre's duality holds for any Weil divisor on $X$. Hence by Serre's 
vanishing, Serre's duality and a similar argument to that in the proof of Corollary 
\ref{3.8}, we have $H^j(X,\OO_X(-\ulcorner p^rH\urcorner))=0$ for any $j<\dim X$ and 
$r\gg 0$. The injection (\ref{es10}) implies that $H^j(X,\OO_X(-\ulcorner H\urcorner))=0$ 
holds for any $j<\dim X$, which implies that $H^j(X,K_X+\ulcorner H\urcorner)=0$ 
holds for any $j>0$ by Serre's duality again.
\end{proof}

\begin{thm}\label{3.5}
Let $X=X(\Delta)$ be a toric variety over $k$. Then there exists
a lifting $\wt{X}$ of $X$ over $W_2(k)$ and a lifting $\wt{F}:\wt{X}
\ra \wt{X}'$ of the relative Frobenius morphism $F:X\ra X'$ such that
the following two conditions hold:
\begin{itemize}
\item[(i)] For any effective Cartier divisor $E$ on $X$, there is a lifting
$\wt{E}\subset \wt{X}$ of $E\subset X$, i.e.\ $X$ is strongly liftable;
\item[(ii)] $\wt{F}$ is compatible with $\wt{E}$ which is taken in (i).
\end{itemize}
\end{thm}

\begin{proof}
First of all, we recall some definitions and notation for toric varieties
from \cite{fu93}. Let $N$ be a lattice of rank $n$, and $M$ the dual lattice
of $N$. Let $\Delta$ be a fan consisting of strongly convex rational
polyhedral cones in $N_\R$, and let $A$ be a ring. In general, we denote
the toric variety associated to the fan $\Delta$ over the ground ring $A$
by $X(\Delta,A)$. More precisely, to each cone $\sigma$ in $\Delta$, there
is an associated affine toric variety $U_{(\sigma,A)}=\spec A[\sigma^\vee
\cap M]$, and these $U_{(\sigma,A)}$ can be glued together to form the toric
variety $X(\Delta,A)$ over $\spec A$. Note that almost all definitions,
constructions and results for toric varieties are independent of the ground
ring $A$, although everything is stated in \cite{fu93} over the complex
number field $\C$.

In our case, $X=X(\Delta,k)$ is the toric variety over $k$ associated to the
fan $\Delta$. Let $\wt{X}=X(\Delta,W_2(k))$. Note that $U_{(\sigma,k)}=\spec
k[\sigma^\vee\cap M]$, $U_{(\sigma,W_2(k))}=\spec W_2(k)[\sigma^\vee\cap M]$
is flat over $W_2(k)$ and $U_{(\sigma,W_2(k))}\times_{\spec W_2(k)}\spec k
=U_{(\sigma,k)}$, hence $\wt{X}$ is a lifting of $X$ over $W_2(k)$.

By definition, to lift the relative Frobenius morphism $F:X\ra X'$ over
$W_2(k)$, we have only to lift the absolute Frobenius morphism $F_X:X\ra X$
over $W_2(k)$. On each affine piece $U_{(\sigma,W_2(k))}$, define $\wt{F}_
\sigma:U_{(\sigma,W_2(k))}\ra U_{(\sigma,W_2(k))}$ by $F_{W_2(k)}:W_2(k)\ra
W_2(k)$ and $\sigma^\vee\cap M\ra \sigma^\vee\cap M$, $u\mapsto pu$. It is
easy to see that we can glue $\wt{F}_\sigma$ together to obtain a morphism
$\wt{F}_X:\wt{X}\ra \wt{X}$ lifting $F_X:X\ra X$.

Let $\LL=\OO_X(E)$ be the associated invertible sheaf on $X$.
By \cite[Page 63, Proposition]{fu93}, we have an exact sequence:
\[
0\ra M\ra \Div_T(X)\ra \Pic(X)\ra 0.
\]
Thus there exists a torus invariant Cartier divisor $D$ on $X$ such that $E$
is linearly equivalent to $D$. Assume that $\{ u(\sigma)\in M/M(\sigma) \}\in
\varprojlim M/M(\sigma)$ determines the torus invariant Cartier divisor $D$.
Then the same data $\{ u(\sigma)\in M/M(\sigma) \}$ also determines a torus
invariant Cartier divisor $\wt{D}$ on $\wt{X}$ (we have only to change the
base $k$ into $W_2(k)$). Thus the invertible sheaf $\LL=\OO_X(E)=\OO_X(D)$
has a lifting $\wt{\LL}=\OO_{\wt{X}}(\wt{D})$ on $\wt{X}$. Let $v_i$ be the
first lattice points in the edges of the maximal dimensional cones in
$\Delta$, $D_i$ the corresponding orbit closures in $X$, and $\wt{D}_i$ the
corresponding orbit closures in $\wt{X}$ ($1\leq i\leq N$). Then the torus
invariant Cartier divisors $D=\sum^N_{i=1}a_iD_i$ and $\wt{D}=\sum^N_{i=1}a_i
\wt{D}_i$ determine a rational convex polyhedral $P_D$ in $M_\R$ defined by
\[
P_D=\{ u\in M_\R\,\,|\,\,\langle u,v_i\rangle\geq -a_i,\,\,1\leq i\leq N \}.
\]
By \cite[Page 66, Lemma]{fu93}, we have
\[
H^0(X,D)=\bigoplus_{u\in P_D\cap M}k\cdot\chi^u,\,\,
H^0(\wt{X},\wt{D})=\bigoplus_{u\in P_D\cap M}W_2(k)\cdot\chi^u.
\]
Thus the map $H^0(\wt{X},\wt{D})\stackrel{r}{\ra}H^0(X,D)$ induced by the
natural surjection $W_2(k)\stackrel{r}{\ra}k$ is obviously surjective.
Hence $r:H^0(\wt{X},\wt{\LL})\ra H^0(X,\LL)$ is surjective.
By Proposition \ref{2.10}, $X$ is strongly liftable over $W_2(k)$.

By construction, assume that $\wt{E}\subset \wt{X}$ is a lifting of
$E\subset X$ such that $\wt{E}$ is linearly equivalent to $\wt{D}$.
By the definition of the lifting $\wt{F}_X:\wt{X}\ra \wt{X}$, we have
$\wt{F}^*_X\OO_{\wt{X}}(-\wt{D})=\OO_{\wt{X}}(-p\wt{D})$ holds, hence
$\wt{F}^*_X\OO_{\wt{X}}(-\wt{E})=\OO_{\wt{X}}(-p\wt{E})$ holds.
\end{proof}

\begin{cor}\label{3.6}
Let $X=X(\Delta,k)$ be a toric variety, and $D$ a reduced Weil
divisor on $X$. Then there is a required open subset $U$ for $(X,D)$,
such that $(U,D|_U)$ has a lifting $(\wt{U},\wt{D|_U})$ over $W_2(k)$
and $F:U\ra U'$ has a lifting $\wt{F}:\wt{U}\ra \wt{U}'$, which is
compatible with $\wt{D|_U}$.
\end{cor}

\begin{proof}
Let $V=X(\Delta_1)$, where $\Delta_1$ is the fan consisting of all
1-dimensional cones in $\Delta$. Then it is easy to see that $V$ is
an open subset of $X$ with $\codim_X(X-V)\geq 2$, $V$ is a smooth toric
variety over $k$, and $D|_V$ is a reduced Cartier divisor on $V$.

By Theorem \ref{3.5}, $(V,D|_V)$ has a lifting $(\wt{V},\wt{D|_V})$ over
$W_2(k)$ and $F:V\ra V'$ has a lifting $\wt{F}:\wt{V}\ra \wt{V}'$, which
is compatible with $\wt{D|_V}$. Shrink $V$ into $U$ if necessary,
such that $U$ is an open subset of $X$ with $\codim_X(X-U)\geq 2$,
$U$ is smooth over $k$ and $D|_U$ is simple normal crossing on $U$.
Then $(U,D|_U)$ has a lifting $(\wt{U},\wt{D|_U})$ over $W_2(k)$
and $F:U\ra U'$ has a lifting $\wt{F}:\wt{U}\ra \wt{U}'$, which is
compatible with $\wt{D|_U}$.
\end{proof}

Now, the main theorems are easy consequences of the above more general
results.

\begin{proof}[Proof of Theorem \ref{1.1}]
It follows from Theorem \ref{3.1} and Corollary \ref{3.6}.
\end{proof}

\begin{proof}[Proof of Theorem \ref{1.2}]
It follows from Theorem \ref{3.2} and Corollary \ref{3.6}.
\end{proof}

\begin{proof}[Proof of Theorem \ref{1.3}]
By Kodaira's lemma, we can take an effective $\Q$-divisor $B$ with sufficiently small 
coefficients such that $H-B$ is ample and $\ulcorner H-B\urcorner=\ulcorner H\urcorner$. 
Therefore, we may from the beginning assume that $H$ is an ample $\Q$-divisor. Now, 
the conclusion follows from Theorem \ref{3.4}, Corollary \ref{3.6} and
toric Serre's duality \cite[Theorem 9.2.10(a) and Remark 9.2.11(b)]{clh}.
\end{proof}

The following corollary shows that the main theorems could hold
for more general varieties which are not necessarily toric.

\begin{cor}\label{3.7}
Let $X$ be a smooth projective toric variety with $\dim X\geq 2$,
and $g:Y\ra X$ the composition of some blow-ups along closed points.
Then the conclusions of Theorems \ref{1.1}, \ref{1.2} and \ref{1.3} 
hold for $Y$.
\end{cor}

\begin{proof}
Let $f:X_1\ra X$ be the blow-up of $X$ along a closed point $P$.
If we can prove the liftability of the relative Frobenius morphism
and the strong liftability of $X_1$, then by Theorems \ref{3.1}, \ref{3.2} and 
\ref{3.4}, the conclusions of Theorems \ref{1.1}, \ref{1.2} and \ref{1.3} hold 
for $X_1$, hence hold for $Y$ by induction.

By Theorem \ref{3.5}, we have a lifting $\wt{X}$ of $X$ over $W_2(k)$
and a lifting $\wt{F}_X:\wt{X}\ra \wt{X}$ of $F_X:X\ra X$ compatible with
any lifting of divisors on $\wt{X}$. By \cite[Lemma 2.5]{xie}, $P\in X$ has
a lifting $\wt{P}\in \wt{X}$. Let $\wt{f}:\wt{X}_1\ra \wt{X}$ be the blow-up
of $\wt{X}$ along $\wt{P}$, $\wt{E}$ the exceptional divisor of $\wt{f}$,
and $\II_{\wt{P}}$ the ideal sheaf of $\wt{P}\in \wt{X}$. Consider the
following diagram:
\[
\xymatrix{
\wt{X}_1 \ar@{.>}[r]^{\wt{F}_{X_1}} \ar[d]_{\wt{f}} & \wt{X}_1
\ar[d]^{\wt{f}} \\
\wt{X} \ar[r]^{\wt{F}_X} & \wt{X}.
}
\]
Since $(\wt{f}\wt{F}_X)^{-1}\II_{\wt{P}}\cdot\OO_{\wt{X}_1}=
\OO_{\wt{X}_1}(-p\wt{E})$ is an invertible sheaf of ideals on $\wt{X}_1$,
by the universal property of blow-up, there is a unique morphism
$\wt{F}_{X_1}:\wt{X}_1\ra \wt{X}_1$ making the above diagram commutative.
It is easy to verify that $\wt{F}_{X_1}:\wt{X}_1\ra \wt{X}_1$ is a lifting of
$F_{X_1}:X_1\ra X_1$ compatible with any lifting of divisors on $\wt{X}_1$.

On the other hand, by \cite[Proposition 2.6]{xie}, $X_1$ is strongly liftable
over $W_2(k)$, which completes the proof.
\end{proof}

For instance, smooth projective rational surfaces satisfy the condition in 
Corollary \ref{3.7}, hence the conclusions of Theorems \ref{1.1}, \ref{1.2} 
and \ref{1.3} hold for smooth projective rational surfaces.

\textsc{School of Mathematical Sciences, Fudan University,
Shanghai 200433, China}

\textit{E-mail address}: \texttt{qhxie@fudan.edu.cn}

\end{document}